\documentclass[10pt, A4, leqno, showkeys]{amsart} 
\date{\today}
\usepackage[active]{srcltx}
\usepackage[]{graphics}
\usepackage{here}
\usepackage{mathtools}
\usepackage[mathscr]{eucal}
\usepackage{mathrsfs}
\usepackage{amscd}
\usepackage{amsfonts}
\usepackage{amsmath,amsthm,amssymb}
\usepackage{latexsym}
\usepackage{comment}
\usepackage{xcolor}
\usepackage[colorlinks=false,linkbordercolor=green,citebordercolor=red]{hyperref}
\usepackage{setspace}

\numberwithin{equation}{section}
\usepackage{bm}

\usepackage{amsthm}
\theoremstyle{plain}
\newtheorem{theorem}{Theorem}[section]
\newtheorem{proposition}[theorem]{Proposition}
\newtheorem{lemma}[theorem]{Lemma}

\newtheorem{fact}[theorem]{Fact}

\theoremstyle{definition}
\newtheorem{definition}[theorem]{Definition}
\newtheorem{example}[theorem]{Example}
\newtheorem{remark}[theorem]{Remark}
 \newtheorem*{acknowledgements}{Acknowledgements}

\newcommand{\inner}[2]{\left\langle{#1},{#2}\right\rangle}
\newcommand{\R}{\boldsymbol{R}}

\newcommand{\Z}{\boldsymbol{Z}}
\newcommand{\C}{\boldsymbol{C}}
\newcommand{\rt}{\sqrt}
\newcommand{\fr}{\frac}

\renewcommand\Re{\operatorname{Re}}
\renewcommand\Im{\operatorname{Im}}
\renewcommand\bar{\overline}
\renewcommand\phi{\varphi}
\renewcommand\Sigma{\varSigma}
\renewcommand\Phi{\varPhi}
\newcommand\ord{\operatorname{ord}}
\newcommand\Hess{\operatorname{Hess}}
\newcommand\tr{\operatorname{tr}}

\renewcommand\tilde{\widetilde}
\renewcommand\epsilon{\varepsilon}
\newcommand{\BSigma}{\bbar{\Sigma}}
\newcommand{\del}{\partial}
\newcommand{\bbar}{\overline}

\newcommand{\al}{\alpha}

\newcommand{\gm}{\gamma}

\newcommand{\Chat}{\C\cup \{\infty\}}

\newcommand{\n}{\bm n}
\newcommand{\s}{\mathcal{S}}

\allowdisplaybreaks

\title[Affine maxface]
{A class of affine maximal surfaces with singularities and its relationship with minimal surface theory}

\author{
Jun Matsumoto
}

\date{\today}



\address{
Department of Mathematics, \endgraf
Institute of Science Tokyo, \endgraf
O-okayama, Meguro, Tokyo, 152-8551, Japan
}
\email{matsumoto.j.d273@m.isct.ac.jp}

\begin{document}

\maketitle

\begin{abstract}
We study a global theory of affine maximal surfaces with singularities, which are called affine maximal maps and defined
by Aledo--Mart\' inez--Mil\' an.
In this paper, we define a special subclass of such surfaces other than improper affine fronts, called \emph{affine maxfaces}, and investigate their global properties with respect to certain notions of completeness.
In particular, by applying Euclidean minimal surface theory, we show that ``complete'' affine maxfaces satisfy an Osserman-type inequality. 
Moreover, one can also observe that affine maxfaces are in a class that does not contain non-trivial improper affine fronts.
We also provide examples of such surfaces which are related to Euclidean minimal surfaces.
\end{abstract}

\section{Introduction}

In affine differential geometry, it is well known that
a locally strongly convex affine minimal surface in unimodular affine 3-space $\R^3$ is locally given by the graph of a function $\phi(x,y)$ defined on a planar domain that satisfies the fourth-order partial differential  equation
\begin{equation}\label{maxeq}
\phi_{yy}\rho_{xx}-2\phi_{xy}\rho_{xy}+\phi_{xx}\rho_{yy} = 0,
\qquad
\rho \coloneqq (\det(\Hess \phi))^{-3/4},
\end{equation}
where $\Hess \phi$ denotes the positive definite Hessian matrix of $\phi(x,y)$.
Blaschke \cite{Bla23} showed that the equation \eqref{maxeq} is equivalent to  everywhere vanishing of the
affine mean curvature.
In this sense, the affine minimal surface is an affine analog of the Euclidean minimal surface 
(cf. \cite{Osserman}, \cite{LM99}, \cite{AFLminimal}).
However, Calabi \cite{Calabi82} showed that the second variation of the affine area functional for affine minimal surfaces is always non-positive (see also \cite{VV89}).
Hence today, the affine minimal surface is called the \emph{affine maximal surface}.

From a global viewpoint of the equation \eqref{maxeq}, the following two conjectures are well known:
\begin{itemize}
\item (The Chern conjecture (\cite{Chern73}, \cite{Chern79}))
\emph{Any affine maximal immersion in $\R^{3}$ expressed by the graph of a strongly convex function defined on the entire plane is equiaffinely equivalent to the elliptic paraboloid.}
\item (The Calabi conjecture (\cite{Calabi82}))
\emph{Any locally strongly convex affine complete affine maximal immersion in $\R^{3}$
is equiaffinely equivalent to the elliptic paraboloid.}
\end{itemize}
For these conjectures, Trudinger and Wang \cite{TW00} solved the Chern conjecture, and Li and Jia \cite{LJ01}
 (see also
 \cite{Calabi58} and \cite{TW02}) proved the Calabi conjecture in the affirmative.
In particular, after Li and Jia's work, the Calabi conjecture has been called the affine Bernstein theorem.
Thus, in global theory, the affine Bernstein theorem has motivated the study of affine maximal surfaces with admissible singularities.

For affine maximal surfaces with singularities,
Mart\'{i}nez \cite{Mart05_IAmap} first studied the improper affine spheres with admissible singularities.
Here, an \emph{improper affine sphere} is a special affine maximal surface that is locally
given by the trivial solution $\phi(x,y)$ of \eqref{maxeq}, 
which satisfies the elliptic Monge--Amp\'ere equation
\begin{equation}
\rho^{-4/3} = \det(\Hess \phi) = \phi_{xx}\phi_{yy}-\phi_{xy}^2 = 1.
\end{equation}
Hence, he generalized the concept of improper affine spheres to \emph{improper affine maps}, which admit certain kinds of singularities, and derived a Weierstrass-type representation formula for these surfaces by using the representation formula for improper affine spheres (\cite{FMM96}, \cite{FMM99}).
Since Nakajo \cite{Nakajo09} proved that any locally strongly convex improper affine map is a wave front,
these surfaces are also referred to as the \emph{improper affine fronts}.
Moreover, Mart\'inez defined the notion of completeness for improper affine fronts like other classes of surfaces with singularities (e.g., the flat surfaces in hyperbolic 3-space (\cite{KUY04}), maximal surfaces in Lorentz--Minkowski 3-space (\cite{UY06_maxface}), constant mean curvature $1$ surfaces in de Sitter 3-space (\cite{Fujimori06}))  and investigated their global properties.
In particular, he showed that complete improper affine fronts satisfy an Osserman-type inequality.
His work led to the application of the Weierstrass-type representation for improper affine fronts, which has been used to clarify various geometric properties of such surfaces up to the present day (\cite{Mart05_Relatives}, \cite{ACG07_Cauchy}, \cite{completeness_lemma}, \cite{KN12}, \cite{Kawa13}, \cite{Milan13}, \cite{MM14},\cite{MM14_geometric_aspects}, \cite{MMT15}, \cite{Kodachi21}, \cite{KK24}, \cite{Matsu24}).

In the more general case, Aledo, Mart\'{i}nez, and Mil\'{a}n (\cite{AMM09_affine_Cauchy}, \cite{AMM09_Non_remo}) firstly studied the affine maximal surface with singularities that are not necessarily improper affine fronts.
Later, in \cite{AMM09_max_map}, they defined a concept of affine maximal surfaces with admissible singularities, 
which are called \emph{affine maximal maps} (Definition \ref{def:maximal map}), 
by extending the following Weierstrass-type representation formula (\cite{Calabi88}) for affine maximal immersions:
\begin{equation}
\psi = 2\Re\left( i\int (\Phi + \bar\Phi) \times d\Phi\right) : \Sigma \to \R^3,
\end{equation}
where $\Phi$ is a (not necessarily single-valued) $\C^3$-valued holomorphic map on a Riemann surface $\Sigma$
satisfying the following \emph{period conditions}:
\begin{equation}
(1)\
N\coloneqq \Phi + \bar \Phi \text{ is single-valued}
\qquad
\text{and}
\qquad
(2)\
\Re\left( i\int_c (\Phi + \bar\Phi) \times d\Phi\right) = 0
\end{equation}
for any closed curve $c$ on $\Sigma$.
The map $N : \Sigma \to \R^3$ is called the \emph{conormal map} of $\psi$.
The Weierstrass-type representation has played important roles in the study of affine maximal surface theory (\cite{Calabi88}, \cite{Li89}, \cite{Li89b}).
In particular, using this formula,
they defined the notion of completeness with respect to the affine metric  for affine maximal maps (Definition \ref{hcomp});
this concept generalizes the classical notion of affine completeness.
They investigated the embeddedness of complete ends, extended the affine Bernstein theorem to the complete regular affine maximal maps with only one embedded end, and characterized such surfaces with only two complete regular embedded ends in \cite{AMM11}. %

On the other hand, when applying the Weierstrass-type representation formula, one must solve the problem: \emph{for given holomorphic data on a Riemann surface, does it satisfy a period condition?}
This problem is often referred to as the \emph{period problem}.
However, the period problem for affine maximal maps differs from that 
for other surfaces which admit the Weierstrass-type representation
because not only is the map $\Phi$ not necessarily single-valued, but also $(\Phi + \bar \Phi)\times d\Phi$ is neither single-valued nor necessarily holomorphic (meromorphic).
Moreover, the method of dealing with a certain kind of Gauss map as a meromorphic function on a Riemann surface is not yet known for affine maximal maps.

This paper aims to define a special class of affine maximal maps, 
which we call \emph{affine maxfaces} (Definition \ref{afmaxfacedef}), 
and to relate the theory of affine maxfaces to that of Euclidean minimal surfaces.
The paper is organized as follows:
In section \ref{preliminary}, we review the fundamental theory of affine immersions and some results for affine maximal maps
by Aledo, Mart\'inez, and Mil\'an.
Next, in section \ref{affine maxface section}, we define the affine maxfaces.
In this class, assuming that the affine conormal map $N$ is a Euclidean conformal minimal immersion in $\R^3$,
one can apply Euclidean minimal surface theory to affine maxfaces.
In particular, the affine Gauss map can be identified with a meromorphic function on a Riemann surface.
Then, one can see that affine maxfaces are locally in a class that excludes non-trivial improper affine fronts (Theorem \ref{localaffinemaxface}).
In section \ref{singularities affine maxface}, we first show that all affine maxfaces are wave fronts (Proposition \ref{maxfacefront}).
Additionally, we give criteria of some kinds of singularities, which are cuspidal edges and swallowtails,
for affine maxfaces (Propositions \ref{maxcusp} and \ref{maxswa})
by applying the criteria of \cite{KRSUY05}.
In section \ref{complete affine maxface}, we study a global property of affine maxfaces.
When we consider them, the two different metrics naturally arise, which are the affine metric and the first fundamental form of the conormal map $N$ as a Euclidean minimal surface.
Hence, in studying global theory, it is natural to consider the relationship between the two kinds of completeness of these metrics.
We say that the completeness of the first fundamental form of $N$ is called \emph{weak completeness} (Definition \ref{weakcompdef}) and find that weak completeness is weaker than the completeness of the affine metric (Lemma \ref{weakcomplem} and Proposition \ref{afcplweakcpl}).
Moreover, we show that an affine maxface that is either complete regular (section \ref{about cpl maxmap}),
or both weakly complete and of finite total curvature, satisfies an Osserman-type inequality (Theorem \ref{OsineqT} and Remark \ref{OsineqR}).
As an application of the inequality, we prove that complete affine maxfaces define a new subclass
of complete affine maximal maps, which does not contain non-trivial complete improper affine fronts
(Theorem \ref{maxfaceIArel}).
Furthermore, it also follows from the proof of Theorem \ref{maxfaceIArel} that 
 any complete affine maxface with constant affine Gauss map is the elliptic paraboloid.
Finally, in section \ref{Examples}, we provide examples of affine maxfaces obtained from
Euclidean minimal surfaces.

\begin{acknowledgements}
The author would like to express his gratitude to Kotaro Yamada for his helpful advice and comments on this research.
This work was supported by JST SPRING, Japan Grant Number JPMJSP2180.
\end{acknowledgements}

\enlargethispage{\baselineskip}

\section{Preliminaries}\label{preliminary}

\subsection{Foundations of affine immersions}\label{1:fundation}

First, we will briefly review some definitions and fundamental facts about the geometry of affine immersions in unimodular affine 3-space $\R^3$ (see \cite{LSZ93}, \cite{LSZH15}, and \cite{NS94} for details). 
Let $\Sigma$ be a connected and orientable
2-manifold, $\psi \colon \Sigma \to \R^3$ an immersion, and $\xi$ a vector field of $\R^3$ along $\psi$ which is transversal to $d\psi(T\Sigma)$. 
Then, there uniquely exist a torsion-free affine connection $\nabla$, a symmetric quadratic form $h$, a $(1,1)$-tensor $S$, and a $1$-form 
$\tau$ on $\Sigma$, which satisfy
\begin{equation}\label{eq:GW}
\left\{
\begin{array}{l}
D_Xd\psi(Y) = d\psi(\nabla_XY) + h(X,Y) \xi,\\
D_X\xi = -d\psi(S(X)) + \tau(X) \xi,
\end{array}
\right.
\end{equation} 
where $D$ is the canonical connection of $\R^3,$ and $X$ and $Y$ are vector fields on $\Sigma$. 
When the quadratic form $h$ is positive definite, the map $\psi$ is said to be \emph{locally strongly convex}. 
Throughout this paper, we only consider the case that $\psi$ is locally strongly convex. 
For a given locally strongly convex immersion $\psi$, 
one can uniquely choose the transversal vector field $\xi$, which satisfies
\begin{equation}
\left\{
\begin{array}{l}\label{xicondition}
D_X\xi = -d\psi(S(X)),\\
\lbrack d\psi(X), d\psi(Y), \xi\rbrack =(h(X,X)h(Y,Y)-h(X,Y)^2)^{1/2},
\end{array}
\right.
\end{equation} 
where $\lbrack \ ,\ ,\ \rbrack$ denotes the determinant function of $\R^3$.
Then, the transversal vector field $\xi$ and the quadratic form $h$ which satisfy \eqref{eq:GW} and \eqref{xicondition} are called the 
\emph{affine normal vector field} and the \emph{affine metric} of $\psi$ with respect to $\xi$ respectively, 
and also the pair $(\psi, \xi)$ (or simply $\psi$)  is called a \emph{Blaschke immersion}. 
For a given Blaschke immersion $\psi : \Sigma \to \R^3$ with an affine normal vector field $\xi$, the  \emph{conormal map} $N : \Sigma \to \R^3$ is defined as an immersion that satisfies the following conditions:
\begin{equation}\label{conormaldef}
\inner N {d\psi} = 0, \qquad 
\inner N \xi = 1,
\end{equation}
where $\inner{\ }{\ }$ denotes the standard Euclidean inner product in $\R^3$.
Moreover, the map $\nu : \Sigma \to \R^3$ defined by
\begin{equation}\label{eq:affine_Gauss_def} 
\nu \coloneqq \fr{\xi}{|\xi|} = \fr{N_u \times N_v}{|N_u \times N_v|},
\end{equation}
where $(u,v)$ is a local coordinate system of $\Sigma$ 
and $|\cdot|$ denotes the standard Euclidean norm on $\R^3$,
is called the (\emph{affine}) \emph{Gauss map} (see \cite[\S 3]{Li89} and \cite[\S 4.2]{LSZ93}), 
This map corresponds to the classical Gauss map of the conormal map $N : \Sigma \to \R^3$. 
Additionally, a Blaschke immersion $\psi: \Sigma \to \R^3$ is called an \emph{affine maximal immersion} if the \emph{affine mean curvature}, which is half of $\tr S$, everywhere vanishes, or equivalently, 
$\Delta_h N = 0$ holds, where $\Delta_h$ is a Laplacian with respect to the affine metric $h$. 
In particular, the Blaschke immersion is said to be an \emph{improper affine sphere} if $S=0$ holds. 
Note that, by definition, any improper affine sphere is an affine maximal immersion.

\subsection{Affine maximal maps and their completeness}\label{about cpl maxmap}
Next, we briefly review the geometry of affine maximal maps (see \cite{AMM09_affine_Cauchy}, \cite{AMM09_max_map}, 
\cite{AMM09_Non_remo}, and \cite{AMM11}).
Firstly, we review the Weierstrass representation of affine maximal immersion (see \cite{Calabi88}).
Let $\psi : \Sigma \to \R^3$ be an immersion with the Euclidean unit normal vector field $N_e$ and the Euclidean second fundamental form $h_e$, and denote by $K_e$ the Euclidean Gaussian curvature of $\psi$. 
Assume that the second fundamental form $h_e$ is definite and the source manifold $\Sigma$ is oriented so that $h_e$ is positive definite. 
Then, we define the positive definite quadratic form $h$, the transversal vector field $\xi$, and the map $N : \Sigma \to \R^3$ by
$
h \coloneqq K_e^{-1/4}h_e,\ 
\xi \coloneqq (1/2)\Delta_h\psi,
$
and
$
N \coloneqq K_e^{-1/4}N_e.
$
Then, the pair $(\psi, \xi)$ defines a locally strongly convex Blaschke immersion with the affine metric $h$ and the conormal map $N$.

We now suppose that a locally strongly convex Blaschke immersion $\psi : \Sigma \to \R^3$ is affine maximal.
Here, we regard the source manifold $\Sigma$ as a Riemann surface whose complex structure is compatible with $h$.
If $\Sigma$ is simply connected, then there exists a holomorphic map $\Phi : \Sigma \to \C^3$ such that
\begin{align}
N &= \Phi + \bar \Phi = 2\Re(\Phi), \label{conormal map}\\  
h &= -2i\left\lbrack \Phi + \bar \Phi, d\Phi, \bar{d\Phi}\right \rbrack, \label{affine metric} \\
\xi &= \fr{\del N \times \bar \del N}{\lbrack N, \del N, \bar \del N \rbrack} 
= \fr{d\Phi \times \bar{d\Phi}}{\left\lbrack \Phi + \bar \Phi, d\Phi, \bar{d\Phi}\right \rbrack}.\label{affine normal}
\end{align}
Here, the exterior derivative $d$ on a Riemann surface splits as $d = \del + \bar \del$, 
where in a local holomorphic coordinate $z$ and 
for a function $f$,
we have $\del f = (\del f/\del z)dz$ and $\bar \del f = (\del f/\del \bar z)d\bar z$. 
By the Lelieuvre formula, we can recover the affine maximal immersion $\psi$ as
\begin{equation}\label{lelif}
 \psi = 2\Re\left( i\int (\Phi + \bar\Phi) \times d\Phi\right) 
= -i\Phi \times \bar\Phi -2\Im\int\Phi\times d\Phi.
\end{equation}

Conversely, let $\Phi : \Sigma \to \C^3$ be a holomorphic map on a Riemann surface $\Sigma$ 
satisfying the following \emph{period conditions}: 
\begin{equation}\label{eq:maxmapperiodcd}
(1)\
N\coloneqq \Phi + \bar \Phi \text{ is single-valued}
\qquad
\text{and}
\qquad
(2)\
\Re\left( i\int_c (\Phi + \bar\Phi) \times d\Phi\right) = 0,
\end{equation}
and 
the quadratic form $-2i\left[\Phi + \bar \Phi, d\Phi, \bar{d\Phi}\right]$ is positive definite. 
Then, the map $\psi : \Sigma \to \R^3$ defined by \eqref{lelif} gives a locally strongly convex 
affine maximal immersion whose conormal map $N$, affine metric $h$, 
and affine normal vector field $\xi$ are given by \eqref{conormal map}, \eqref{affine metric}, and \eqref{affine normal},
respectively.

By the above arguments, we get a complex representation formula of affine maximal immersions 
from harmonic functions or a holomorphic map.
However, when the affine metric $h$ in \eqref{affine metric} vanishes at a point, 
the map $\psi$ in \eqref{lelif} has a singular point there. 
Hence, Aledo, Mart\'{i}nez, and Mil\'{a}n defined a class of affine maximal surfaces
with admissible singularities as follows:

\begin{definition}[\cite{AMM09_max_map}, \cite{AMM11}]\label{def:maximal map}
Let $\Sigma$ be a Riemann surface. 
We say that a map $\psi : \Sigma \to \R^3$ is an \emph{affine maximal map} 
if there exists a holomorphic map $\Phi : \tilde \Sigma \to \C^3$, where $\tilde \Sigma$ stands for the 
universal cover of $\Sigma$, such that
$\left[\Phi + \bar \Phi, d\Phi, \bar{d\Phi}\right]$ is not identically zero, and $\psi$ is given by
\eqref{lelif}.
We also call the map $\Phi$ the \emph{Weierstrass data}, $N \coloneqq \Phi + \bar \Phi$ the \emph{conormal map}, and $h\coloneqq -2i\left[\Phi + \bar \Phi, d\Phi, \bar{d\Phi}\right]$ the \emph{affine metric}.
\end{definition}

The singular point of an affine maximal map corresponds to the point where the affine metric vanishes. 
In addition, as referred to in \cite[Remark 3]{AMM11}, 
if the image of the conormal map is contained in a plane in $\R^3$, then the affine maximal map
is an improper affine front (\cite{Mart05_IAmap}).

\begin{definition}[\cite{AMM09_max_map}, \cite{AMM11}]\label{hcomp}
The affine maximal map $\psi:\Sigma\to\R^3$ with the affine metric $h$ is \emph{complete} if
there exists symmetric $(0,2)$-tensor $T$ with compact support such that $|h|+T$ is a complete Riemannian metric.
\end{definition}

\noindent
This definition extends the notion of affine completeness to possibly degenerate affine metrics, following the approach of \cite{KUY04}.

\begin{fact}[\cite{AMM11}]\label{Humaxmap}
Let $\psi : \Sigma \to \R^3$ be a complete affine maximal map.
The Riemann surface $\Sigma$ is biholomorphic to $\BSigma\setminus\{p_1, \dots, p_n\}$ with  $n\geq1$,
where $\BSigma$ is a compact Riemann surface and $\{p_1, \dots, p_n\}\subset\BSigma$. 
\end{fact}
\noindent
Each puncture point $p_j$ (or the image of a small neighborhood of it) is called an \emph{end} of the affine maximal map. 
In addition, we call an end $p$  a \emph{regular end} if the differential $d\Phi$ extends meromorphically to $p$, 
and if not, we call the end $p$ an \emph{irregular end}. 
Moreover, we say that a complete affine maximal map is \emph{regular} if all ends are regular ends. 

\begin{remark}
Note that the Gaussian curvature $K_h$ 
with respect to the affine metric $h = -2i\left[\Phi + \bar \Phi, d\Phi, \bar{d\Phi}\right]$ of a locally strongly convex affine maximal immersion $\psi : U\ (\subset \C) \to \R^3$  is non-negative.
In fact, $K_h$ is computed as follows: 
We locally write the affine metric as $h = e^{2\sigma(z)}|dz|^2$, 
where $\sigma \coloneqq (1/2) \log( -2i[\Phi + \bar \Phi, \Phi_z, \bar{\Phi_z}])$.
Then,
\begin{align*}
K_h &= -e^{-2\sigma}\varDelta \sigma 
\qquad \left(\varDelta \coloneqq 4 \fr{\del^2}{\del z \del \bar z}\right)\\
&= - \fr{2\left(\left[\Phi + \bar \Phi, \Phi_{zz}, \bar{\Phi_{zz}}\right]\left[\Phi + \bar \Phi, \Phi_z, \bar{\Phi_z}\right]
-\left[\Phi + \bar \Phi, \Phi_{zz}, \bar{\Phi_z}\right]\left[\Phi + \bar \Phi, \Phi_z, \bar{\Phi_{zz}}\right]\right)}
{-2i \left[\Phi + \bar \Phi, \Phi_z, \bar{\Phi_z}\right]^3}\\
&= \fr{\left|\left[\Phi + \bar \Phi, \Phi_{zz}, \Phi_z\right]\right|^2}{i\left[\Phi + \bar \Phi, \Phi_z, \bar{\Phi_z}\right]^3}.
\end{align*}
The third equality follows from the following identity (see \cite[Lemma 4.2.3.10]{LSZ93}):
$$
[\bm a, \bm b, \bm c][\bm d, \bm b, \bm e] - [\bm a, \bm b, \bm e][\bm d, \bm b, \bm c]
= [\bm a, \bm b, \bm d][\bm c, \bm b, \bm e] 
\qquad
\left(\bm a, \bm b, \bm c, \bm d, \bm e \in \R^3\ \left(\text{or} \  \C^3\right)\right).
$$
Since $h$ is positive definite, one can see that $-2i[\Phi + \bar \Phi, \Phi_z, \bar{\Phi_z}]^2 > 0$,
and so $i[\Phi + \bar \Phi, \Phi_z, \bar{\Phi_z}]^3 > 0$. 
Hence, we get $K_h \geq 0$.
\end{remark}

Aledo, Mart\' inez, and Mil\' an \cite{AMM11} also showed the following:
\begin{fact}[\cite{AMM11}]\label{ExafB}
\begin{enumerate}
\item \label{embeddedend}
An end $p$ of a complete regular affine maximal map with Weierstrass data $\Phi = (\Phi_1, \Phi_2, \Phi_3)$ is an embedded end if and only if the $\C^3$-valued meromorphic $1$-form $d\Phi$ has a pole of order $2$ at $p$ and there exists $j\in\{1,2,3\}$ such that $d\Phi_j$ has a pole of order less than $2$ at $p$.
\item\label{ExafB1}
(Extended affine Bernstein theorem)
The elliptic paraboloid is the only complete regular affine maximal map with a unique embedded end.
\end{enumerate}
\end{fact}

In the last of this section, we describe an example of an affine maximal maps
given by Aledo, Mart\'inez, and Mil\'an.

\begin{example}{\cite[Section 4]{AMM11}}\label{ex:220}
Let $\Sigma = \C \setminus\{0\}$ and 
\begin{equation}
\Phi = (\Phi_1, \Phi_2, \Phi_3) =  \left(a_1z + \fr{b_1}{z} + c_1,\, a_2z + \fr{b_2}{z} + c_2,\,  k \log z + c_3\right).
\end{equation}
If we choose parameters  $a_1, a_2, b_1, b_2, c_1, c_2, c_3, k\in \C$ to satisfy 
the period condition \eqref{eq:maxmapperiodcd},  we obtain a family of affine maximal maps 
with two complete regular embedded ends at $z = 0$ and $\infty$.
When we set $F_1  \coloneqq \Phi_1, F_2 \coloneqq \Phi_2$ and $F_3 \coloneqq \Phi_3 - k \log$,
the degrees of $F_j\ (j = 1,2,3)$ are $\deg(F_1) = \deg(F_2) = 2, \deg(F_3) = 0$
as holomorphic maps from the Riemann sphere $\C \cup \{\infty\}$ to itself.
In this sense, the above family is called \emph{$220$-type} (Figure \ref{Figure220}).
More generally, they similarly defined families of complete regular affine maximal maps 
that satisfy
$$
0 \leq \deg(F_j) \leq 2, \qquad
\sum^3_{j = 1} \deg(F_j) = 4 \qquad
(j = 1,2,3).
$$
Each $F_j$ is defined on a compact Riemann surface of any genus minus two points, 
and the resulting families are called \emph{canonical examples}.
They also showed that a complete affine maximal map with exactly two regular embedded ends is a canonical
example (\cite[Theorem 15]{AMM11}).
\end{example}

\begin{figure}[h]
\begin{center}
\begin{tabular}{c}
\includegraphics[height=45mm]{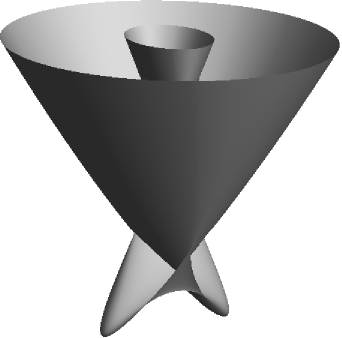}
\end{tabular}
\caption{Affine maximal map of $220$-type ($a_1 = -2i, a_2 = 2, b_1 =1, b_2 = i, c_j = 0, k = 1/2$)}\label{Figure220} 
\end{center}
\end{figure}

\section{Affine Maxfaces}\label{affine maxface section}
In this section, we introduce a special class of affine maximal maps and investigate its relationship with the Euclidean minimal surface theory (see \cite{Osserman}, \cite{LM99} and \cite{AFLminimal}).

\begin{definition}\label{afmaxfacedef}
Let $\psi : \Sigma \to \R^3$ be an affine maximal map on a Riemann surface $\Sigma$ with the affine conormal map $N=\Phi + \bar \Phi $ as in Definition \ref{def:maximal map}.
We say that  the affine maximal map $\psi$ is the \emph{affine maxface} if the conormal map $N$ is a Euclidean 
conformal minimal immersion in $\R^3$.
We also call the conormal map $N$ the \emph{minimal conormal map}.
\end{definition}

\noindent
The key point of Definition \ref{afmaxfacedef} is that the conormal map is a conformal immersion.
This is because, on the set of regular points, a map to be affine maximal is equivalent to the conormal map
to be harmonic with respect to the affine metric as described in Section \ref{1:fundation}.
The term ``maxface" was initially coined for maximal surfaces with singularities in Lorentz–Minkowski 3-space in \cite{UY06_maxface}.

Now, let us consider an affine maxface $\psi : \Sigma \to \R^3$ 
defined on a simply connected Riemann surface $\Sigma$.
Then, since its minimal conormal map $N = \Phi + \bar \Phi$ is a Euclidean conformal minimal immersion in $\R^3$, 
it can be represented by the classical Weierstrass representation for Euclidean minimal surfaces 
(\cite[Lemmas 8.1 and 8.2]{Osserman} and \cite[Theorem 2.3.4]{AFLminimal}) as follows:
\begin{align}
N &= \Phi + \bar \Phi  = \Re \int (1-g^2, i(1+g^2), 2g) \omega \quad : \Sigma \to \R^3,  \label{conomini}\\
\Phi &=  \fr{1}{2} \int (1-g^2, i(1+g^2), 2g) \omega \quad :\Sigma \to \C^3, \label{nullcurve}
\end{align}
where $g$ is a meromorphic function, and $\omega$ is a holomorphic $1$-form on $\Sigma$.
In addition, the Euclidean first fundamental form $d\sigma^2$ of $N$ is 
\begin{equation}\label{ffofN}
d\sigma^2 = (1+|g|^2)^2|\omega|^2.
\end{equation}
On the other hand, the Euclidean unit normal vector field $\nu$ of the minimal conormal map $N$ is given by
\begin{equation}
\nu = 
\fr{N_z \times N_{\bar z}}{|N_z \times N_{\bar z}|}
= 
\left(
\fr{2\Re(g)}{1+|g|^2}, \fr{2\Im(g)}{1+|g|^2}, \fr{-1+|g|^2}{1+|g|^2} 
\right)
= \Pi \circ g,
\end{equation}
where the map $\Pi : \C\cup\{\infty\} \to S^2$ is the stereographic projection from the north pole of the unit sphere $S^2$ centered at the origin in $\R^3$.
The map $\nu : \Sigma \to S^2$ corresponds to the (affine) Gauss map 
(see \eqref{eq:affine_Gauss_def}). 
Hence, we also refer to the meromorphic function $g$ the (\emph{affine}) \emph{Gauss map} of the affine maxface.
Note that, in general, even if the affine Gauss map of a Blaschke immersion is constant, the immersion is not 
necessarily trivial.
In fact, the improper affine spheres always admit the constant affine Gauss map.   
Moreover, the affine metric $h$ is given by 
\begin{equation}\label{afmetmaxface}
h=-2i[\Phi + \bar\Phi, d\Phi, \bar{d\Phi}]=\inner{N}{\nu}d\sigma^2,
\end{equation}
and the singular point corresponds to the point where  it holds that 
\begin{equation}\label{sing1}
\inner{N}{\nu}=0.
\end{equation}

\begin{remark}
When constructing an affine maximal map, we must consider the period conditions \eqref{eq:maxmapperiodcd}
if the source Riemann surface $\Sigma$ is not simply connected.
For affine maxfaces, those period conditions are equivalent to 
\begin{equation}
\Re \int_{C_1} (1-g^2, i(1+g^2), 2g) \omega=0\\
\end{equation}
and
\begin{equation}\label{maxperiod}
\Re\left( i\int_{C_2} (\Phi + \bar\Phi) \times d\Phi\right)=0,
\end{equation}
where $C_1$ and $C_2$ are any closed curves in $\Sigma$. 
Thus, to construct an affine maxface,  one needs to choose complex constants of the integration in \eqref{nullcurve} to satisfy \eqref{maxperiod}.
\end{remark}

A trivial example of an affine maxface is as follows:

\begin{example}[Elliptic paraboloid]\label{ex:elliptic paraboloid}
Let $\Sigma  = U\subset \C$ be a simply connected domain, and let  
\begin{equation}
g=0,\qquad  \omega =dz.
\end{equation} 
Then, the map $\Phi$ is given by 
$$
\Phi = \fr{1}{2}(z+c_1,\,  iz + c_2,\,  c_3) \qquad (c_1, c_2, c_3 \in \C),
$$
and the minimal conormal map $N:\Sigma\to\R^3$ gives the plane as the Euclidean minimal immersion.
It is easy to check that the affine maxface given by this $\Phi$ is equiaffinely equivalent to a piece of  the elliptic paraboloid
for any choice of $c_1, c_2,$ and $ c_3$.
\end{example}

At the end of this section, we will prove that, locally, any affine maxface that is also an improper affine front must be trivial.

\begin{theorem}\label{localaffinemaxface}
Let $U \subset \C$ be a simply connected domain and $\psi : U \to \R^3$ an affine maxface. 
The affine maxface is an improper affine front if and only if the image $\psi(U)$ 
is contained in an elliptic paraboloid.
\end{theorem}

\begin{proof}
Since the ``if'' part is clear,
we have to show the ``only if'' part.
Assume that $\psi : U \to \R^3$ is an improper affine front. 
Without loss of generality, we may choose the affine normal vector field $\xi$ as $\xi \equiv (0,0,1)$ on $U$. 
By \cite{Mart05_IAmap}, there exists a holomorphic regular curve $(F, G) :U \to \C^2$ (i.e., $F$ and $G$ are holomorphic functions on $U$ satisfying $(dF, dG) \neq (0,0)$), which is the Weierstrass data of 
$\psi$ as the improper affine front, such that $N=(\bar F-G, 1)\in \C\times \R=\R^3$. 
Hence, since the image $N(U)$ lies on a plane $\{(x^1,x^2,x^3)\in\R^3\, ;\,  x^3=1\}$, the Gauss map $g$ is constant. 
Thus, without loss of generality, we may set $g \equiv 0$ on $U$.

On the other hand, since the metric $d\sigma^2 =(1 + |g|^2)^2 |\omega|^2 =  |\omega|^2 
$ is positive definite, 
by changing a coordinate on $U$, we can take $\omega = dz$.
Therefore, from \eqref{lelif}, the pair $(g,\omega) = (0, dz)$ locally gives the elliptic paraboloid (see Example \ref{ex:elliptic paraboloid}).
\end{proof}

\section{Singularities of affine maxfaces}\label{singularities affine maxface}

In this section, we discuss singularities of affine maxfaces.
In \cite{KRSUY05}, the criteria for singular points of a front in
$\R^3$ to be cuspidal edges or swallowtails were given.
The differential geometry of surfaces with singularities is studied in \cite{SUY09_front} and \cite{yamadasing} in detail.

As a first step, let us recall fundamental concepts and notations of singularities of smooth maps.
Let $f : U \to \R^3$ be a smooth map defined on a domain $U$ in the $uv$-plane $\R^2$.
The map $f$ is said to be a \emph{frontal} if there exists a smooth map
$\bm n :U \to S^2$ such that $\bm n(p)$ is perpendicular to $(df)_p(T_pU) \subset \R^3$ for each $p \in U$.
We call $\bm n$ the \emph{unit normal vector field} of $f$. 
Moreover, for a given frontal $f :U \to \R^3$ with a unit normal vector field $\bm n:U\to S^2$ of $f$,
we say that the frontal $f$ is a (\emph{wave}) \emph{front} (resp. (\emph{wave}) \emph{front} at $p\in U$) if the pair of smooth maps 
$$
(f, \bm n) : U \to \R^3 \times S^2
$$
is an immersion on $U$ (resp. immersion at $p\in U$).
Clearly, an immersion is a front.

Now let $f : U\to\R^3$ be a frontal with a normal vector field $\bm N$, which is not necessarily unit.
A smooth function $\Lambda : U \to \R$ defined by 
\begin{equation}\label{eq:lambda}
\Lambda \coloneqq [f_u, f_v, \bm N]
\end{equation}
is called an \emph{identifier of singularities}.
By definition, a point $p \in U$ is a singular point of $f$ if and only if $\Lambda(p) = 0$.
As a function obtained by multiplying $\Lambda$ by a non-vanishing smooth function 
can play the same role as $\Lambda$, 
it is also called an \emph{identifier of singularities}.
A singular point $p\in U$ is called \emph{non-degenerate} if
$d\Lambda$ does not vanish at $p$.
We assume  $p$ is a non-degenerate singular point.
Then, by the implicit function theorem, there exists a regular curve
$$
c=c(t):(-\epsilon,\epsilon)\to U \qquad (\epsilon > 0)
$$
(called the \emph{singular curve}) such that $c(0)=p$ and the image
of $c$ coincides with the set of singular points of $f$ around $p$.
On the other hand, a nowhere vanishing vector field $\eta$ along the singular curve $c$ such that
$(df)_{c(t)}(\eta(t)) = 0$ for each $t \in (-\epsilon, \epsilon)$ is called a \emph{null vector field}.
Along the singular curve $c(t)$, the null vector field $\eta(t)$ is uniquely determined up to multiplication
by nowhere vanishing functions.
The direction of the vector $\eta(t)$ for each $t$ is called the \emph{null direction} at $c(t)$.
Under these settings, the following criteria for singularites are well known:

\begin{fact}[\cite{KRSUY05}]\label{KRSUYcri}
Let  $p \in U$ be a non-degenerate singular point of a front
$f:U\to \R^3$, $c: (-\epsilon, \epsilon) \to U$ the singular curve with $c(0) = p$,
and $\eta(t)$ a null vector field.
\begin{enumerate}
\item \label{cuspidalfact}
The front $f$ has a cuspidal edge at $p$ 
if and only if  $[\dot c(0), \eta(0)] \neq 0$ ($\dot c\coloneqq dc/dt$), 
where $[\ ,\ ]$ stands for the determinant in $\R^2$.
\item \label{swallowfact}
The front $f$ has a swallowtail at $p$ 
if and only if $[\dot c(0), \eta(0)] = 0$ 
and
$$
\left.\frac{d}{dt}\left[\dot c(t),\eta(t)\right] \right|_{t=0}\neq 0.
$$
\end{enumerate}
\end{fact}

Now let us study the singularities of affine maxfaces.
Let $U\subset \C = \R^2$ be a domain with the standard complex coordinate 
$z = u + iv$, $\psi :  U \to \R^3$ an affine maximal map (not necessarily an affine maxface) with the conormal map 
$N=(N_1, N_2, N_3)=\Phi + \bar{\Phi} : U \to \R^3$.
We also denote the singular set of $\psi$ by $\mathcal{S} \subset U$. 
In addition, note that all affine maximal maps are frontals 
since we can take a unit normal vector field $\n \coloneqq N/|N|$.

\begin{lemma}\label{Phiimmersion}
If the holomorphic map $\Phi : U \to \C^3$ is an immersion and $(d\psi)_p \neq 0$ at $p\in \s$, 
then the affine maximal map $\psi$ is a front there.
\end{lemma}

\begin{proof}
By \eqref{lelif}, the derivative of $\psi$ can be computed as
\begin{equation}\label{dpsi}
d\psi
=
i\left\{
(\Phi+\bar \Phi)\times \Phi_zdz - (\Phi+\bar \Phi)\times \bar{\Phi_z} d \bar z
\right\}.
\end{equation}
Let $\bm v = v_1(\del/\del z)_p + \bar v_1(\del/\del \bar z)_p \in \text{Ker}(d\psi)_p \setminus \{0\} \subset T_p U
\ (v_1 \in \C\setminus \{0\})$. 
Here, we identify the tangent space $T_pU$ of $U$ at $p$ with $\R^2$ and $\C$ as 
\begin{equation}\label{identi}
\zeta=a+i b \in \C
\leftrightarrow (a,b)\in\R^2
\leftrightarrow a \left(\fr{\partial}{\partial u}\right)_p + b \left(\fr{\partial}{\partial  v}\right)_p
\leftrightarrow \zeta \left(\fr{\partial}{\partial z}\right)_p + \bar \zeta \left(\fr{\partial}{\partial \bar z}\right)_p
\in T_pU.
\end{equation}
Then, by \eqref{dpsi}, the vector $\bm v$ must satisfy $v_1(\Phi + \bar \Phi)\times \Phi_z \in\R^3$.
We will show $dN(\bm v) \nparallel N$ at $p$.
If $dN(\bm v) \parallel N$ at $p$, then direct computation gives that $(\Phi_z \times \bar{\Phi_z})(p) = \bm 0$,
which contradicts the assumption that $\Phi$ is an immersion at $p$. 
Hence, we observe $dN(\eta)\nparallel N$ at $p$.
By \cite[Lemma 1.6]{FSUY08}, we see that the affine maximal map $\psi$ is a front at $p$.
\end{proof}

Hereafter, let us consider an affine maxface $\psi :  U \to \R^3$ with the minimal conormal map
$N: U \to \R^3$, and let $(g,\omega = \hat \omega(z)dz)$ be the pair as in \eqref{conomini}, where $\hat \omega$ is a holomorphic function on $U$.
Note that $p \in \mathcal S$ if and only if 
\begin{equation}\label{sing2}
N_1(g + \bar g) - iN_2 (g-\bar g) + N_3(-1+|g|^2) =0
\end{equation}
holds at $p$ by \eqref{sing1}.

\begin{proposition}\label{maxfacefront}
For any $p \in \s$, if $(d\psi)_p = 0$, then it holds that $|g(p)| \neq 1$.
Moreover, any affine maxface is a front at each $p \in \s$.  
\end{proposition}

\begin{proof}
We carry out this proof at the point $p \in \s$.
Let us show the former part by contradiction.
Suppose that $|g| = 1$.
Then, the equation \eqref{sing2} is equivalent to
\begin{equation}
iN_1\left(\fr{1}{g}+g\right) = N_2\left(\fr{1}{g}-g\right)
\end{equation}
by $|g| = 1$. 
And then, by \eqref{dpsi}, we have
\begin{align*}
d\psi 
&=
\fr{i}{2}\left(
\left(
2gN_2-iN_3(1+g^2),N_3(1-g^2)-2gN_1, iN_1(1+g^2)-N_2(1-g^2)
\right)\omega
\right.\\
&
-
\left.
\left(
2\bar g N_2+iN_3(1+\bar g^2), N_3(1-\bar g^2)-2\bar g N_1, -iN_1(1+\bar g^2)-N_2(1-\bar g^2)
\right)\bar \omega
\right)\\
&= 
\fr{i}{2}\left(
\left(
2N_2-iN_3\left(\bar g+g\right), -N_3\left(g-\bar g\right)-2N_1, iN_1\left(\fr{1}{g}+g\right)
-N_2\left(\fr{1}{g}-g\right)
\right)g\omega
\right.\\
&
-
\left.
\left(
2 N_2+iN_3\left(\bar g+ g\right),  N_3\left(g - \bar g\right)-2 N_1, -iN_1\left(\fr{1}{g}+ g\right)
+N_2\left(\fr{1}{g}- g\right)
\right)\bar{g\omega}
\right)\\
&=
-2\Im\left(
\left(
N_2-iN_3\Re(g), -iN_3\Im(g)-N_1,0
\right)g\omega
\right).
\end{align*}
Since $d\psi =0$, it holds that 
\begin{align*}
(N_2-iN_3\Re(g))g\hat \omega
&=
(-iN_3\Im(g)-N_1)g\hat \omega 
=0,\\
(N_2+iN_3\Re(g))\bar{g\hat \omega}
&=
(iN_3\Im(g)-N_1)\bar{g\hat \omega} = 0. 
\end{align*}
Hence, we get 
$$
\left(N_2^2+N_3^2\Re(g)^2\right)|g|^2|\hat \omega|^2 = 0.
$$
Since $|g| =1$ and $\hat\omega \neq 0$ because of the positive definiteness of the metric $d\sigma^2 = (1+|g|^2)^2|\omega|^2 = 4|\omega|^2$, we obtain 
$N_2 = N_3\Re(g) =0$. 
Similarly, we get $N_1 = N_3\Im(g) = 0$, and thus $N_1 = N_2 = N_3 =0$ holds,
which contradicts that $N \neq 0$ by \eqref{conormaldef}. 
Thus, if $(d\psi)_p = 0$, then we have $|g| \neq 1$ at $p\in\s$.  

We are to show the latter part.
If $(d\psi)_p \neq 0$, then the affine maxface is a front at $p$ by Lemma \ref{Phiimmersion}.
Let us assume $(d\psi)_p = 0$. 
For the unit normal vector field $\bm n = N/|N|$ of $\psi$, direct computations
yield that  
$$
\bm n_z \times \bm n_{\bar z} 
=
\fr{1}{|N|}\nu - \fr{2i}{|N|^3}\Im(N_z\times N) = \fr{1}{|N|}\nu \neq 0
$$
because of $\Im(N_z\times N) = 0$.
In fact, if $d\psi$ = 0, then $|g| \neq 1$ by the formar part.
By \eqref{sing2}, we observe $N_3 = (-N_1(g + \bar g) + iN_2(g - \bar g))/(|g|^2 - 1)$.
Thus we obtain
\begin{multline}\label{dpsin3nai}
d\psi 
=
\fr{i}{2}
\left(\fr{2\Re(g)}{|g|^2-1}, \fr{2\Im(g)}{|g|^2-1}, 1\right) \\
\cdot\left(
\left(iN_1(1+g^2)-N_2(1-g^2)\right)\hat \omega dz
-
\bar{\left(iN_1(1+g^2)-N_2(1-g^2)\right)\hat \omega dz}
\right),
\end{multline}
and thus $(iN_1(1+g^2)-N_2(1-g^2))\hat \omega =0$ holds.
Thus, we find that this equation yields $\Im (N_z \times N) = 0$.
Hence, the unit normal vector field $\bm n : U \to S^2$ is an immersion at $p$, and so is $(\psi, \bm n) : U \to \R^3\times S^2$.

Therefore, the affine maxface $\psi: U \to \R^3$ is a front at $p\in \s$.
\end{proof}

\begin{lemma}\label{lem:nondege}
A singular point $p \in \mathcal S$ of an affine maxface is non-degenerate if and only if $((N_1(1-g^2) + i N_2(1+g^2) + 2g N_3)dg)(p) \neq 0$.
\end{lemma}

\begin{proof}
One can compute $\psi_u \times \psi_v$ as follows:
\begin{equation}
\psi_u \times \psi_v = -2i\psi_z\times \psi_{\bar z} 
= \inner{N}{\nu}(1+|g|^2)^2|\hat \omega|^2 N.
\end{equation}
Then, one can take an identifier of singularity $\Lambda$ as $\Lambda = \inner{N}{\nu}$.
Now, assume that the equation \eqref{sing2} holds. 
Since $\Lambda_z = \inner{N}{\nu_z}$ because of $\inner{N_z}{\nu} = 0$ and
\begin{equation}\label{nuz}
\nu_z 
=
\fr{g_z}{2(1+|g|^2)^2}
\left(1-\bar g^2, -i(1+\bar g^2), 2\bar g\right),
\end{equation}
we get 
$$
d\Lambda = 2\Re(\Lambda_z dz)
= \fr{\Re (N_1(1-g^2) + i N_2(1+g^2) + 2g N_3)dg)}{(1+|g|^2)^2} 
$$
Hence, a point $p\in \mathcal{S}$ is non-degenerate 
if and only if $ ((N_1(1-g^2) + i N_2(1+g^2) + 2g N_3)dg)(p) \neq 0$.
\end{proof}

\begin{proposition}\label{maxcusp}
Let $p\in \s$ be a non-degenerate singular point of an affine maxface.
Then, the affine maxface has a cuspidal edge at a singular point $p\in\s$ if and only if 
\begin{equation}
((N_1(1-g^2) + i N_2(1+g^2) + 2g N_3)dg)(p) \neq 0
\end{equation}
and 
one of the following conditions holds: 
\begin{enumerate}
\item
if $|g(p)| \neq 1$, then 
\begin{equation}\label{maxfacecusp1}
\Im\left(
\left(N_1(1+g^2)+iN_2(1-g^2)\right)^2\bar{g_z}\hat \omega
\right)(p)
\neq 0.
\end{equation}
\item \label{maxfacecusp2}
If $|g(p)|=1$, then either of the following holds:
\begin{enumerate}
\item 
if $\Re(g(p)) \neq 0$, then
\begin{equation}\label{maxfacecusp2a}
\Im\left(
(N_2-iN_3\Re(g))^2\bar{g_z}g^2\hat \omega \right)(p) \neq 0,
\end{equation}
or 
\item 
if $\Im(g(p)) \neq 0$, then 
\begin{equation}\label{maxfacecusp2b}
\Im\left(
(N_2+iN_3\Im(g))^2\bar{g_z}g^2\hat \omega \right)(p) \neq 0.
\end{equation}
\end{enumerate}
\end{enumerate}
\end{proposition}

\begin{proof}
We discuss this proof at $p\in \s$.
By \eqref{dpsi}, one can compute $d\psi$ as 
\begin{multline}\label{dpsi2}
d\psi
=
\fr{i}{2}\left(
\left(
2gN_2-iN_3(1+g^2),N_3(1-g^2)-2gN_1, iN_1(1+g^2)-N_2(1-g^2)
\right)\omega
\right.\\
-
\left.
\left(
2\bar g N_2+iN_3(1+\bar g^2), N_3(1-\bar g^2)-2\bar g N_1, -iN_1(1+\bar g^2)-N_2(1-\bar g^2)
\right)\bar \omega
\right).
\end{multline}
Let $c(t)$ the singular curve with $c(0) = p$.
\begin{enumerate}
\item \label{cusppf:cuspcri1}
Assume $|g| \neq 1$.
Under the identification in \eqref{identi}, by \eqref{dpsin3nai}, 
$$
\eta = \bar{(iN_1(1+g^2)-N_2(1-g^2))\hat \omega}
$$
gives a null direction at $p$.
Differentiating the both sides of $\inner{N}{\nu}(c(t)) = 0$ by $t$, we have
$$
0 = \fr{d}{dt}(\inner{N}{\nu}(c(t)))
= \Re\left(\inner{N}{\nu_z}(c(t))\dot c(t)\right).
$$
This implies that $\dot c(t)$ is perpendicular to $\inner{N}{\nu_z}$ at $c(t)$, that is, proportional to 
$i\bar{\inner{N}{\nu_z}}$. 
Thus, by \eqref{nuz}, one can parametrize $c$ as 
\begin{equation} \label{cdot}
\dot c(t) = i\bar{\inner{N}{\nu_z}} = -\fr{i \bar{g_z}(N_1(1+g^2)+iN_2(1-g^2))}{(1+|g|^2)(-1+|g|^2)}.
\end{equation}
at each $t \in (-\epsilon, \epsilon)$.
Here, by direct calculations, we have 
$$
\left[\dot c, \, \eta\right]
=\Im\left(\bar{\dot c} \, \eta\right) = 
-\fr{\Im\left(
\left(N_1(1+g^2)+iN_2(1-g^2)\right)^2\bar{g_z}\hat \omega
\right)}{2(1+|g|^2)(-1+|g|^2)}.
$$
Hence, by \eqref{cuspidalfact} in Fact \ref{KRSUYcri}, we have \eqref{maxfacecusp1}.

\item \label{cusppf:cuspcri2}
We assume $|g| =1$.
If $\Re(g) \neq 0$, then \eqref{sing2} is rewritten as
$$
N_1 = -\fr{\Im(g)}{\Re(g)}N_2.
$$
By the computation of $d\psi$ in the proof of Proposition \ref{maxfacefront}, 
\begin{equation}
d\psi = \fr{i}{\Re(g)}
\left((N_2-iN_3\Re(g))g\omega - \bar{(N_2-iN_3\Re(g))g\omega}\right)(\Re(g), \Im(g), 0).
\end{equation}
Hence, 
$$
\eta = \bar{(N_2-iN_3\Re(g))g\hat \omega}
$$
defines the null-direction at $p$. 
In addition, by \eqref{cdot} and \eqref{sing2}, we can parametrized $c(t)$ satisfying
\begin{align*}
\dot c(t) &= \fr{g\bar{g_z}(N_2-iN_3\Re(g))}{4\Re(g)},\\
[\dot c, \eta] = -\Im\left(\dot c \, \bar \eta\right)
&=
-\fr{\Im\left((N_2-iN_3\Re(g))^2\bar{g_z}g^2 \hat \omega\right)}{4\Re(g)},
\end{align*}
and by Fact \ref{KRSUYcri} again, get \eqref{maxfacecusp2a}.
Similarly, we have \eqref{maxfacecusp2b}.
\end{enumerate}
\end{proof}%

\begin{proposition}\label{maxswa}
Let $p\in \s$ be a non-degenerate singular point of an affine maxface.
Then, the affine maxface has a swallowtail at a singular point $p\in\s$ if and only if 
\begin{equation}\label{swa:nondeg}
(g_z \hat \omega\inner{N}{N_{\bar z}})(p)\neq 0
\end{equation}
and 
one of the following conditions holds: 
\begin{enumerate}
\item
If $|g(p)| \neq 1$, then 
\begin{equation}\label{swa:notcusp1}
\Im\left(
\left(N_1(1+g^2)+iN_2(1-g^2)\right)^2\bar{g_z}\hat \omega
\right)(p)
= 0,
\end{equation}
and 
\begin{equation}\label{swa:cri1}
\Re\left(
\left(
\left(N_1(1+g^2)+iN_2(1-g^2)\right)^2 \bar{g_z}\hat \omega
\right)_z
\left(N_1(1+g^2)+iN_2(1-g^2)\right)g_z
\right)(p)
\neq 0.
\end{equation}
\item
If $|g(p)|=1$, then either of the following holds: 
\begin{enumerate}
\item 
if $\Re(g(p)) \neq 0$, then
\begin{equation}
\Im\left(
(N_2-iN_3\Re(g))g^2\bar{g_z}\hat \omega \right)(p) = 0,
\end{equation}
and
\begin{equation}\label{swa:cri2a}
\Im\left(
\left(\left(N_2-iN_3\Re(g)\right)g^2\bar{g_z}\hat \omega\right)_z g \bar{g_z}
\right)(p)
\neq 0
\end{equation}
or 
\item 
if $\Im(g(p)) \neq 0$, then 
\begin{equation}\label{swa:notcusp2b}
\Im\left(
(N_2+iN_3\Im(g))g^2\bar{g_z}\hat \omega \right)(p) = 0.
\end{equation}
and
\begin{equation}\label{swa:cri2b}
\Im\left(
\left(\left(N_1+iN_3\Im(g)\right)g^2 \bar{g_z}\hat \omega\right)_z g \bar{g_z}
\right)(p)
\neq 0.
\end{equation}
\end{enumerate}
\end{enumerate}
\end{proposition}

\begin{proof}
Let $p \in \s$ satisfy \eqref{swa:nondeg}. 
We now discuss this proof at $p$. 

(1) 
Assume $|g| \neq 1$ and \eqref{swa:notcusp1}.
Then, from the proof \eqref{cusppf:cuspcri1} in Proposition \ref{maxcusp}, we have
\begin{align*}
\left.\fr{d}{dt}\right|_{t=0} [\dot c, \, \eta]
&=
-\fr{\Im\left(\left(
\left(N_1(1+g^2)+iN_2(1-g^2)\right)^2\bar{g_z}\hat \omega
\right)_z \dot c
\right)}{2(1+|g|^2)(-1+|g|^2)}\\
&= 
\fr{\Re\left(\left(
\left(N_1(1+g^2)+iN_2(1-g^2)\right)^2\bar{g_z}\hat \omega
\right)_z (N_1(1+g^2)+iN_2(1-g^2))
\right)}{4(1+|g|^2)^2(-1+|g|^2)^2}\\
\end{align*}
and obtain \eqref{swa:cri1} by Fact \ref{KRSUYcri}.

The argument in (1) proceeds in exactly the same way as that in (2),
and thus leads to the conclusion.
\end{proof}

\section{Complete Affine Maxfaces}\label{complete affine maxface}

In this section, we shall investigate global properties of complete affine maxfaces by applying 
Euclidean minimal surface theory.
To begin with, we examine the relationship between two notions of completeness:
that of the affine metric $h$ in \eqref{afmetmaxface} and that of the first fundamental form $d\sigma^2$ of the minimal conormal map in \eqref{ffofN}.

In the case of a general affine maximal map, Aledo--Mart\'inez--Mil\' an already defined completeness
and regularity in Definition \ref{hcomp} and after Fact \ref{Humaxmap}.
Thus, we say that an affine maxface is \emph{complete} (resp. \emph{regular}) if it is complete 
(resp. regular) as an affine maximal map.

\begin{lemma}\label{weakcomplem}
If an affine maxface is complete, the metric $d\sigma^2$ is complete.
\end{lemma}

\begin{proof}
Assume that an affine maxface $\psi : \Sigma \to \R^3$ is complete. 
Then, it holds that $\inner{N}{\nu}>0$ outside a compact subset in $\Sigma$. 
Hence, we have
$$
d\sigma^2=\fr{1}{\inner{N}{\nu}}h
$$
and observe that $d\sigma^2$ is complete since $h$ is complete outside the compact set.
\end{proof}

\begin{definition}\label{weakcompdef}
An affine maxface is said to be \emph{weakly complete} if the Riemannian metric $d\sigma^2$ is complete.
\end{definition}

\noindent
We note that weak completeness does not necessarily imply completeness (see Example \ref{Enneper}).

\begin{lemma}\label{HOs}
If a complete affine maxface is regular, then the pair $(g,\omega)$ in \eqref{conomini} can be extended meromorphically to the ends, and  the metric $d\sigma^2$ is of finite total curvature. 
\end{lemma}

\begin{proof}
Since the affine maxface $\psi : \Sigma \to \R^3$ is complete, by Fact \ref{Humaxmap}, 
$\Sigma$ is biholomorphic to $\BSigma\setminus\{p_1,\dots,p_n\}$ under the same notation as the fact. 
And since the affine maxface is regular, 
the $1$-form $d\Phi$ is extended meromorphically to each end $p_j\ (j=1,\dots,n),$ 
and so is the pair $(g,\omega)$. 
In particular, since the Gauss map $g$ is a holomorphic map from $\BSigma$ to $\Chat$,  by the same arguments as the theory of Euclidean minimal surfaces
(see \cite[Theorem 9.2]{Osserman} and \cite[Corollary 2.6.6]{AFLminimal}), 
the total curvature with respect to $d\sigma^2$ is 
\begin{equation}\label{tc}
\int_\Sigma K_{d\sigma^2}dA_{d\sigma^2} = -\int_\Sigma \fr{2i dg \wedge d\bar g}{(1+|g|^2)^2}=-4\pi\deg (g)
\in -4\pi \Z_{\geq0},
\end{equation}
where $K_{d\sigma^2}$ and $dA_{d\sigma^2}$ denote the Gaussian curvature and the area element of $d\sigma^2$, and $\deg(g)$ is the degree as the holomorphic map $g : \BSigma \to \Chat$. 
Therefore, the total curvature with respect to $d\sigma^2$ is finite.
\end{proof}

We call the integral in the left-most side of \eqref{tc} the \emph{total curvature} of an affine maxface. 
In addition, we say that an affine maxface is of \emph{finite total curvature}
if the total curvature is finite; 
in other words, the minimal conormal map $N: \Sigma \to \R^3$
as the conformal minimal immersion is complete and of finite total curvature.

\begin{proposition}\label{afcplweakcpl}
An affine maxface is complete regular if and only if it is weakly complete, of finite total curvature,
and the singular set is compact.
\end{proposition}

\begin{proof}
By Definition \ref{hcomp}, Lemmas \ref{weakcomplem} and \ref{HOs}, if an affine maxface is complete regular, then we see that the metric $d\sigma^2$ is complete (i.e., the affine maxface is weakly complete), of finite total curvature, and  the singular set is compact. 

Conversely, suppose that the singular set $\mathcal{S}$ is compact, $d\sigma^2$ is complete, and of finite total curvature. 
We will argue outside $\mathcal{S}$ hereafter.
Without loss of generality, we may suppose that $\inner{N}{\nu}>0$.
Then, there exists $\epsilon>0$ such that $\inner{N}{\nu}>\epsilon$.
Hence, it holds that
$$
h= \inner{N}{\nu}d\sigma^2 > \epsilon d\sigma^2,
$$
and one can observe that $h$ is complete. 
That is, the affine maxface is complete.
On the other hand, by \cite[Lemmas 9.5 and 9.6]{Osserman} (see also \cite{Huber57}), $\Sigma$ is biholomorphic to $\BSigma\setminus\{p_1, \dots, p_n\}$, where $\BSigma$ stands for a compact Riemann surface, and the pair $(g, \omega)$ can be extended meromorphically to each $p_j$ ($j=1, \dots, n$).
Therefore, it follows that the $1$-form $d\Phi$ is meromorphic on $\Sigma$,  which yields the conclusion.  
\end{proof}

Note that even if the minimal conormal map $N$ is complete and of finite total curvature (i.e., so is $d\sigma^2$), the affine maxface associated with $N$ is not necessarily complete (see Example \ref{Enneper}).%

If we assume that an affine maxface is either complete regular, or weakly complete and of finite total curvature, 
then the minimal conormal map is a complete Euclidean conformal minimal immersion of finite total curvature.
Thus, under these assumptions, we can apply the results on complete minimal surfaces of 
finite total curvature to affine maxfaces.

For example, if an affine maxface is complete regular, then applying the original Osserman inequality
 (\cite[Theorem 9.3]{Osserman}), the Jorge--Meeks formula (\cite[Theorem 4]{JM83} and \cite[Theorem 5.4]{LM99}) 
(see also \cite{Sch83}, \cite{Fang99} and \cite{KUY02}) to the minimal conormal map $N$, which is a complete conformal minimal immersion of finite total curvature, 
and \eqref{embeddedend} in Fact \ref{ExafB}, 
we have the following Osserman-type inequality for complete regular affine maxfaces.

\begin{theorem}[Osserman-type inequality]\label{OsineqT}
Let $\psi : \Sigma=\BSigma_\gm\setminus \{p_1,...,p_n\}\to \R^3$ be a complete regular affine maxface, 
where $\BSigma_\gm$ denotes a compact Riemann surface with genus $\gm$. 
Then, the following inequality holds:
\begin{equation}\label{Osineq}
-\fr{1}{2\pi}\int_\Sigma K_{d\sigma^2}dA_{d\sigma^2}=2\deg(g)\geq -\chi(\BSigma_\gm)+2n,
\end{equation}
where $\chi(\BSigma_\gm) = 2-2\gm$ denotes the Euler characteristic of $\BSigma_\gm$. 
Moreover, the equality of \eqref{Osineq} holds if and only if all ends are embedded.
\end{theorem}

\begin{remark}\label{OsineqR}
The inequality \eqref{Osineq} also holds for weakly complete affine maxfaces of finite total curvature. 
\end{remark}

As an application of the Osserman-type inequality, 
the following theorem states that the complete affine maxface defines a new subclass of complete affine maximal maps. 
We know that this subclass does not contain any non-trivial complete improper affine fronts.

\begin{theorem}\label{maxfaceIArel}
A complete affine maxface $\psi : \Sigma=\BSigma_\gm \setminus \{p_1,\dots, p_n\} \to \R^3$ is an improper affine front if and only if it is the elliptic paraboloid. 
Moreover, any complete affine maxface with constant affine Gauss map is the elliptic paraboloid.
\end{theorem}

\begin{proof}
Since the elliptic paraboloid is a complete improper affine front, it suffices to prove the converse.

Assume that $\psi : \Sigma \to \R^3$ is an improper affine front. 
Note that all complete improper affine fronts are always regular (\cite[Proposition 1]{Mart05_IAmap}).
Without loss of generality, we can take the affine normal vector field $\xi$ as $\xi\equiv(0,0,1)$. 
As in the proof of Theorem \ref{localaffinemaxface},
without loss of generality, we can set $g \equiv 0$ on $\Sigma$.
In other words, the affine Gauss map $g$ is constant.

On the other hand, since $\omega$ is a meromorphic $1$-form on $\BSigma_\gm$ by Lemma \ref{HOs} and has no zeros in $\BSigma_\gm$ by completeness and positive definiteness of $d\sigma^2=(1+|g|^2)^2|\omega|^2=|\omega|^2$, it must have poles at all ends $p_1, \dots, p_n$. 
Hence, by the Riemann--Roch theorem, 
$$
2\gm-2=-\chi(\BSigma_\gm)=\sum_{p\in\omega^{-1}(0)}\ord_p\omega+\sum_{p\in\omega^{-1}(\infty)}\ord_p\omega 
=\sum_{p\in\omega^{-1}(\infty)}\ord_p\omega,
$$
where $\omega^{-1}(0)$ and $\omega^{-1}(\infty)$ are the set of poles and the set of zero points 
of $\omega$ in $\BSigma_\gm$, respectively.
Since the rightmost side above is non-positive, we get $\gm \leq 1$.

Suppose that $\gm=1$. The Osserman-type inequality (Theorem \ref{OsineqT}) leads
$$
2\cdot 0\geq -(2-2\cdot 1) + 2n,
$$
so we have $n\leq0$. This contradicts $n \geq1$. 
Hence, we have $\gm=0$.
By the Osserman-type inequality again, we conclude $n\leq1$ and obtain $n=1$ 
since $n \geq 1$ by Fact \ref{Humaxmap}. 
Then, since the equality in the Osserman-type inequality holds, the unique end $p_1$ is the embedded end. 

In conclusion, the extended affine Bernstein theorem (\eqref{ExafB1} in Fact \ref{ExafB}) yields that the complete regular affine maxface $\psi$, being the improper affine front, must be the elliptic paraboloid. 
The latter assertion is clear from the proof.
\end{proof}

\section{Examples} \label{Examples}

\begin{example}[Enneper-type affine maxface]\label{Enneper}
Let 
\begin{equation}
\Sigma =  \C,\qquad  g=z,\qquad  \omega =dz.
\end{equation} 
The map $\Phi$ is given by
$$
\Phi = 
\fr{1}{2} \left(z-\fr{z^3}{3} + c_1,\,  i \left(\frac{z^3}{3}+z\right) + c_2, \, z^2 + c_3 \right)  \qquad (c_1, c_2, c_3 \in \C),
$$
and the minimal conormal map $N:\Sigma\to\R^3$ gives the Enneper surface as the Euclidean minimal surface.
This $(g, \omega)$ gives the well-defined weakly complete affine maxface of finite total curvature $-4\pi$. 
However, it is incomplete regardless of the choice of $c_1, c_2$, and $c_3$. 
Indeed, the singular set $\{z \in \C;\, S_1(z) = 0\}$, where
\begin{multline*}
S_1(z) \coloneqq4 \Re(z) \Re\left(c_1-\fr{1}{6} z (z^2-3)\right)-4 \Im(z) \Im\left(\fr{1}{6}z\left(z^2+3\right)
-i c_2\right)\\
   +2 \left(|z|^2-1\right) \Re\left(c_3+\fr{z^2}{2}\right),
\end{multline*}
accumulates at the end $z = \infty$.
This fact is proven as follows.
Let $z = r e^{i \theta}\, (r > 0, \theta\in[0, 2\pi))$ and we have 
$S_1(z) = S_1(r, \theta) = (1/3)r^4 \cos (4\theta) + P(r, \theta)$, where 
$$
P(r, \theta) \coloneqq r^2(\cos(2\theta) + 2 \Re(c_3)) + 4r(\Re(c_1)\cos (\theta) + \Re(c_2) \sin(\theta)) - 2\Re(c_3).
$$
One can observe there exist $k_0, k_1 \in \R_{\geq0},$ and $k_2 \in \R_{>0}$ such that 
$|P(r, \theta)| \leq k_2 r^2 + k_1 r + k_0$.
We set $\theta_\delta ^- \coloneqq \pi/4 - \delta,\, \theta_\delta ^+ \coloneqq \pi/4 + \delta\  (0 < \delta < \pi/4)$, 
and $m \coloneqq (1/3) \sin (2\delta) > 0$.
In addition, let 
$$
R \coloneqq \max \left\{1, \rt{\fr{6k_2}{m}}, \rt[3]{\fr{6k_1}{m}}, \rt[4]{\fr{6k_0}{m}}\right\},
$$
and direct computations give that $S_1(r, \theta_\delta^-) > 0$ and $S_1(r, \theta_\delta^+) < 0$ for any $r > R$. 
Moreover, we define a continuous function $T: [\theta_\delta^-, \theta_\delta^+] \to \R$  
by $T(\theta) \coloneqq S_1(r_0, \theta)$ for a fixed number $r_0 >R$.
Then, there exists a number $\theta_0 \in  [\theta_\delta^-, \theta_\delta^+]$ such that $T(\theta_0) = 0$.
Let us take a sequence $(z_n)_{n\in \Z_{\geq 1}} \subset \C$ defined as $z_n \coloneqq (n+R) e^{i \theta_0}$.
Then, we have $|z_n| \to \infty$ as $n\to \infty$ and  $S_1(z_n) = 0$.
Therefore, the end $z = \infty$ is an accumulation point of the singular set.

Figure \ref{Figure1} shows the picuture of an Enneper-type affine maxface with chosen as $c_1 = c_2 = c_3 = 0$. 
In this case, the singular set $\s$ is $\s = \{z \in \C; \Re(z) = \pm \Im(z)\}$.
By Lemma \ref{lem:nondege} and Proposition \ref{maxcusp}, we see that 
$z = 0$ is the unique degenerate singular point, while all other points in $\s \setminus \{0\}$ are
cuspidal edges. 
\end{example}

\begin{figure}[h]
\begin{center}
\begin{tabular}{c@{\hspace{1.8cm}}c}
\includegraphics[height=38mm]{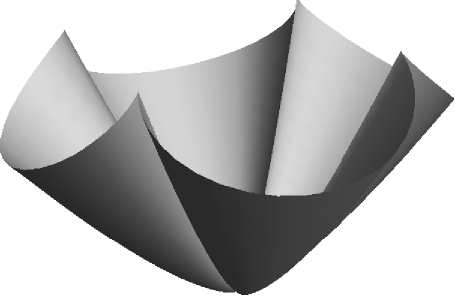}&
\includegraphics[height=40mm]{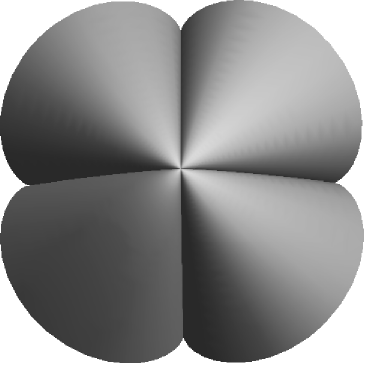}\\

\end{tabular}
\caption{Example \ref{Enneper} : Enneper-type ($c_1 = c_2 = c_3 = 0$) }\label{Figure1}
\end{center}
\end{figure}

\begin{example}[Catenoid-type affine maxface]\label{catenoid}
Let 
\begin{equation}
\Sigma = \C\setminus\{0\},\qquad  g=z,\qquad  \omega =\fr{dz}{z^2}.
\end{equation} 
Then, the map $\Phi$ is given by 
$$
\Phi = \fr{1}{2}\left(-z -\fr{1}{z} + c_1,\, i\left(z -\fr{1}{z} \right) + c_2, \, 2\log z + c_3\right) \qquad (c_1, c_2, c_3 \in \C),
$$
and the minimal conormal map $N:\Sigma\to\R^3$ gives the catenoid as the complete Euclidean minimal immersion.
We must consider the period condition \eqref{maxperiod} in this case.
The affine maxface $\psi = (\psi_1, \psi_2, \psi_3)$ in \eqref{lelif} is given as follows:
\begin{align*}
\psi_1
&=
2(\Re (c_3)+\log |z|)\Re\left(z - \fr{1}{z}\right)-2\Re\left(z + \fr{1}{z}\right) -4(\Re (c_2))(\arg z)\\
\psi_2
&=   
2(\Re (c_3)+\log |z|)\Im\left(z + \fr{1}{z}\right)-2\Im\left(z - \fr{1}{z}\right) +4(\Re (c_1))(\arg z)\\
\psi_3
&=
\fr{1}{2}\left(|z|^2 - \fr{1}{|z|^2}\right) +2 \log|z|-2(\Re (c_1))\Re\left(z-\fr{1}{z}\right) 
 -2(\Re (c_2))\Re\left(z-\fr{1}{z}\right)
\end{align*}
up to an additive constant vector of $\R^3$.
Hence, we deduce that the affine maxface $\psi$ is well-defined if and only if $\Re (c_1) = \Re (c_2) = 0$.
Therefore, the affine maxface $\psi$ is weakly complete and of total curvature $-4\pi$.
Furthermore, the singular set $\{z \in \C\setminus \{0\}; \, S_2(z) = 0\}$, where
$$
S_2(z) \coloneqq (|z|^2-1)(\Re(c_3)+2\log |z|)- |z|^2-1,
$$
is compact for any $c_3$.
In fact, we consider the real valued function $s_2(x)\coloneqq (x^2-1)(k+2\log x) - x^2-1$ defined on $(0, \infty)$.
It satisfies $\lim_{x\to+0}s_2(x) = +\infty$ for any $k\in\R$, $s_2(x)>0$ for any $k\geq 0$ and some sufficiently 
large $x$, 
and $\lim_{x\to +\infty}s_2(x) = +\infty$ for any $k<0$.
Hence, the ends $z = 0$ and $\infty$ are not accumulation points of $\s$.
Therefore, we find that $\psi$ is complete regular by Proposition \ref{afcplweakcpl}.

On the other hand, if we set $z = e^{s + i t}\, (s \in \R, t \in [0, 2\pi))$, then $\psi = \psi(s,t)$ can be rewriten as 
\begin{multline*}
\psi(s,t) = (4(s + \Re(c_3)) \sinh(s) - \cosh(s)) \cos(t),\\ 4((s + \Re(c_3))\sinh(s) - \cosh(s)) \sin (t),\, 
 \sinh(2s) + 2s)
\end{multline*}
up to parallel translation in $\R^3$.
Hence, this surface $\psi$ is a surface of revolution whose profile curve $c$ is 
$
c(s)\coloneqq (4((s + \Re(c_3)) \sinh(s) - \cosh(s)), \sinh(2s) + 2s).
$
One can easily check that, as a planar curve, the profile curve always intersects the vertical axis at two points for any $c_3$.
Thus, the singularities of the catenoid-type affine maxface are cone-like singularities.

Note that this example is contained in the family of the affine maximal maps of $220$-type (Example \ref{ex:220})
 (Figure \ref{Figure2}).
\end{example}


\begin{figure}[h]
\begin{center}
\begin{tabular}{c@{\hspace{2.8cm}}c}
\includegraphics[height=60mm]{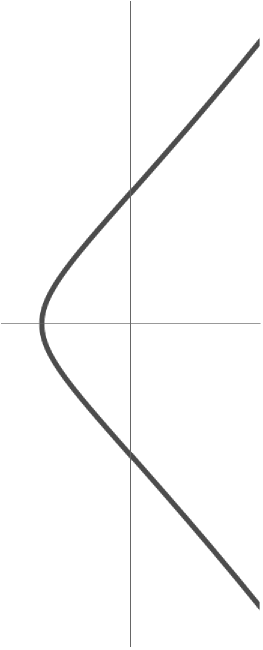}&
\includegraphics[height=60mm]{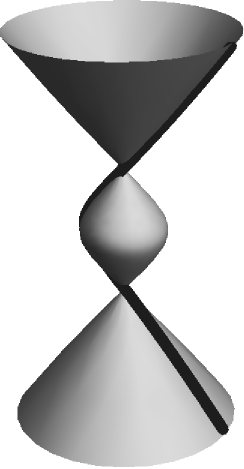}\\
The profile curve $c$  
& Catenoid-type
\end{tabular}
\caption{Example \ref{catenoid} ($c_1 = c_2 = c_3 = 0$)}\label{Figure2}
\medskip
\shortstack{{\small
Left: the profile curve 
$
c(s)= (4((s + \Re(c_3)) \sinh(s) - \cosh(s)), \sinh(2s) + 2s).
$
}
\\
{\small
Right: the profile curve is shown in the black line.
}
}
\end{center}
\end{figure}

\begin{example}[Helicoid-type affine maxface]\label{ex:helicoid}
Let 
\begin{equation}
\Sigma = \C,\qquad  g=e^z,\qquad  \omega =ie^{-z}dz,
\end{equation} 
and then the map $\Phi$ is given by
$$
\Phi = \left(-i \cosh(z) + c_1, -\sinh(z) + c_2, iz + c_3  \right).
$$  
Then, the minimal conormal map $N:\Sigma\to\R^3$ gives the helicoid as the Euclidean minimal immersion.
And then, the affine maxface $\psi$ is always well-defined and of infinite total curvature.
Additionally, the helicoid-type affine maxface is always incomplete. 
In fact, the singular set is equal to	 $\{z \in \C ; S_3(z) = 0\}$, where
$$
S_3(z) \coloneqq \left(e^{2\Re(z)} -1\right)(\Re(c_3) - \Im (z)) + 2\Re\left((\Re(c_1) - i \Re(c_2)) e^z \right).
$$
Setting $z = x + i y \, (x, y \in \R)$, we can rewrite $S_3(z) = S_3(x, y)$ as 
$$
S_3(x, y) = (e^{2x} - 1)(a_3 - y) + 2e^x(a_1 \cos(y) + a_2 \sin(y)),
$$
where $a_j \coloneqq \Re(c_j)\, (j = 1,2,3)$.
Let $y_1 \coloneqq a_3 - \pi$ and $y_2 \coloneqq a_3 + \pi$ and take $M > 0$
so that $|a_1\cos(a_3) + a_2 \sin(a_3)|\leq M$.
Then, one can verify that 
$
S_3(x, y_1) \geq \pi(e^{2x} - 1) - 2Me^x 
$
and 
$
S_3(x, y_2) \leq -(\pi(e^{2x} - 1) - 2Me^x ).
$
Hence, for a sufficiently large $x$, it holds that $S_3(x, y_1) > 0$ and $S_3(x, y_2) < 0$.
Thus, there exists $y_0 \in [y_1, y_2]$ such that $S_3(x, y_0) = 0$.
Letting $z_n = n + i y_0$, we see that $|z_n| \to \infty$ as $n\to \infty$ and $S(z_n) = 0$ for any
sufficiently large $n \in \Z_{\geq 1}$.
Therefore, the singular set accumulates at the end $z = \infty$, which implies that
the helicoid-type affine maxface is incomplete
 (Figure \ref{Figure2}). 

\end{example}

\begin{example}[Affine maxface induced from the minimal M\"{o}bius strip]\label{ex:Mobius}
Let 
\begin{equation}
\Sigma = \C\setminus\{0\},\qquad  g=z^2\fr{z+1}{z-1},\qquad  \omega =i\fr{(z-1)^2}{z^4}dz, 
\end{equation} 
and anti-holomorphic involution $I = - 1/\bar z$. 
The Weierstrass data $\Phi$ is given by
\begin{multline*}
\Phi 
= \left(
\fr{i}{2}\left(
z + \fr{1}{z} +z^2 - \fr{1}{z^2} + \fr{1}{3}\left(z^3 + \fr{1}{z^3}\right)\right) + c_1,\right.\\
\left.\fr{1}{2}\left(
-z + \fr{1}{z} - z^2 - \fr{1}{z^2} + \fr{1}{3}\left(-z^3 + \fr{1}{z^3}\right)\right) + c_2,\, 
i\left(z + \fr{1}{z}\right) + c_3
\right).
\end{multline*}
Then, the $I$-invariant minimal conormal map $N:\Sigma\to\R^3$ induces the non-orientable minimal immersion 
$\Sigma/\langle I \rangle = \R P^2 \setminus \{\text{1 point}\} \to \R^3$,  which is the minimal M\"{o}bius strip as the Euclidean minimal surface (see \cite{Meeks81}).
On the other hand, since the residues of $d\Phi$ always vanish, i.e., the map $\Phi$ is single valued on $\Sigma$,  
the affine maxface is well-defined 
if and only if the residue of $\Phi \times d\Phi$ at $z = 0$ is pure imagenary.
By direct computations, it is equal to $(2i, 0, -2i/3)$.
Thus, the affine maxface is well-defined
for any $c_1, c_2, c_3 \in \C$.
The singular set is given by $\{z \in \C \setminus \{0\}; (|A(z)|^2 - 1)(\Re(c_3) - \Im(z + 1/z)) + \Re(A(z) B(z)) = 0\}$, where
\begin{align*}
A(z) &\coloneqq \fr{z^2(z+1)}{z - 1}, \\
B(z) &\coloneqq 2(\Re(c_1) - i\Re(c_2)) 
+ i\left(z + \fr{1}{z} + z^2 - \fr{1}{z^2} +\fr{1}{3}\left(z^3 + \fr{1}{z^3}\right)\right)\\
&\hspace{4cm}
+ \left(z - \fr{1}{z} + z^2 + \fr{1}{z^2} +\fr{1}{3}\left(z^3 - \fr{1}{z^3}\right)\right).
\end{align*}
Hence, this surface is weakly complete, but incomplete since
the singular set always accumulates at the ends $z = 0$ and $\infty$, regardless of the choice of $c_1, c_2$, and $c_3$,
by an argument similar to that in Example \ref{ex:helicoid}.
Note that although the total curvature of the minimal M\"{o}bius strip as a non-orientable minimal surface is 
$-6\pi$, the resulting affine maxface has the total curvature $-12\pi$.
 (Figure \ref{Figure3}). 

\end{example}

\begin{figure}[h]
\begin{center}
\begin{tabular}{c@{\hspace{0.8cm}}c}
\includegraphics[height=55mm]{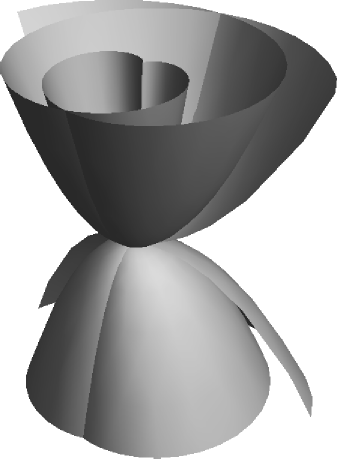}&
\includegraphics[height=55mm]{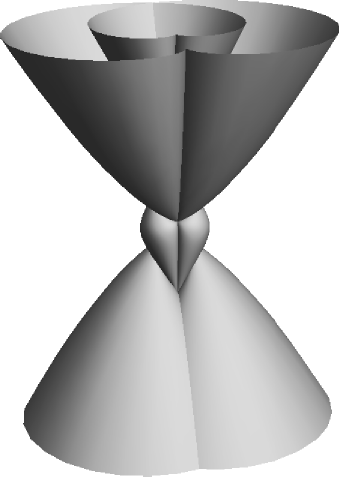}\\
Helicoid-type & Example obtained from the minimal M\"{o}bius strip
\end{tabular}
\caption{Examples \ref{ex:helicoid} and \ref{ex:Mobius} (both $c_1 = c_2 = c_3 = 0$)}\label{Figure3} 
\end{center}
\end{figure}

\begin{example}\label{ex:ms}
Let 
\begin{equation}\label{thm:ms}
\Sigma = \C \setminus \{0\},
\qquad
g = \fr{z^n-a}{z^n-1}, 
\qquad
\omega = \fr{(z^n-1)^2}{z^m}dz,
\end{equation}
where $m, n \in \Z_{\geq 2}$ and $a\in \C\setminus\{0, \pm1\}$  satisfy $n > m-1>1$ and $n\neq 2(m-1)$.

For $n \geq 2$, the pair \eqref{thm:ms} defines a complete Euclidian conformal minimal immersion 
whose Gauss map $g$ omits two values $\{a,1\}$ from the Riemann sphere, 
which was given by Miyaoka--Sato \cite{MS94}.
Let us check the well-definedness of affine maxface.
The equation \eqref{nullcurve} gives that
\begin{align*}
\Phi_1 &= \fr{1}{2} \left(\fr{a^2-1}{m-1}z^{-m+1} + \fr{2(a-1)}{n-m+1}z^{n-m+1}\right) + c_1,\\
\Phi_2 &= \fr{i}{2}\left( \fr{a^2+1}{1-m}z^{-m+1}   - \fr{2(a+1)}{n-m+1}z^{n-m} + \fr{2}{2n-m+1}z^{2n-m+1}\right) + c_2,\\
\Phi_3 &= \fr{a}{1-m}z^{-m+1}-\fr{a+1}{n-m+1}z^{n-m+1}  + \fr{1}{2n-m+1}z^{2n-m+1}  + c_3,
\end{align*}
where $\Phi = (\Phi_1, \Phi_2, \Phi_3)$.
Hence, since the map $\Phi: \Sigma \to \C^3$ is single-valued by $n > m-1 > 1$, 
the affine maxface $\psi  : \Sigma \to \R^3$ is 
well-defined if and only if the residue at the end $z = 0$ of the meromorphic 1-form 
$\al = (\al_1, \al_2, \al_3) \coloneqq \Phi \times d\Phi$ is in $i\R^3$ by \eqref{lelif}.
We can computes $\al$ as follows:
\begin{align*}
\al_1
&=
(F^1_1 z^{-m} +F^2_1 z^{-m+n}  + F^3_1 z^{-m+2n} +F^4_1 z^{-2m+n+1} + F^5_1 z^{-2m+2n+1})dz\\
\al_2 
&=
(F^1_2 z^{-m} +F^2_2 z^{-m+n}  + F^3_2 z^{-m+2n} +F^4_2 z^{-2m+n+1} + F^5_2 z^{-2m+2n+1} \\
& \hspace{8cm}+F^6_2 z^{-2m+3n+1})dz\\
\al_3
&= (F^1_3 z^{-m} +F^2_3 z^{-m+n} +  F^3_3 z^{-m+2n} +F^4_3 z^{-2m+n+1} + F^5_3 z^{-2m+2n+1} \\
& \hspace{8cm}+F^6_3 z^{-2m+3n+1})dz, 
\end{align*}
where $F_j^i\, (i = 1,\dots, 6, \, j = 1,2,3)$ denote rational expressions of $a, m, n, c_1, c_2,$ and $c_3$. 
Since $n > m-1 > 1$ and $n\neq 2(m-1)$, we know that the exponents of $z$ in $\al$ are not equal to $-1$,
and hence the residue of $\al$ at $z = 0$ is $(0,0,0)$.

Thus, we find that the affine maxface given from \eqref{thm:ms} is well-defined, weakly complete, and of total curvature $-4n\pi$.
However, it is incomplete since
the singular set 
$\{z \in \C \setminus\{0\}; S_5(z) = 0\}$, where
\begin{multline*}
S_5(z) \coloneqq 
\Re\left(
\fr{z^n - a}{z^n - 1}
\right)
\left(
a_1 + 
\Re
\left(
\fr{a^2-1}{m-1}z^{-m+1} + \fr{2(a-1)}{n-m+1} z^{n-m+1} 
\right)
\right)\\
+
\Im\left(
\fr{z^n - a}{z^n - 1}
\right)
\left(
a_2 + 
\Im
\left(
\fr{a^2+1}{1-m}z^{-m+1} - \fr{2(a+1)}{n-m+1} z^{n-m+1}\right.\right.\\
\left.\left.\hspace{8cm}
+ \fr{2}{2n-m+1} z^{2n-m+1}
\right)
\right)\\
+
\left(
\left|\fr{z^n - a}{z^n - 1}\right|^2 - 1
\right)
\left(
a_3 + 
\Re
\left(
\fr{a}{1-m}z^{-m+1}
- \fr{a+1}{n-m+1} z^{n-m+1}\right.\right.\\
\left.\left.
+ \fr{1}{2n-m+1} z^{2n-m+1}
\right)
\right)
\end{multline*}
and $a_k = \Re(c_k)\ (k = 1,2,3)$,
accumulates at the end $z = 0$ for any parameters
by an argument similar to that in Example \ref{ex:helicoid}, and thus the resulting affine maxface 
is not complete.
On the other hand, if we choose $a_1 = \Re(c_1) \neq 0$, then the singular set does not accumulate at $\infty$
(Figure \ref{Figure4}).

\end{example}

\begin{figure}[h]
\begin{center}
\begin{tabular}{c@{\hspace{2cm}}c}
\includegraphics[width=45mm]{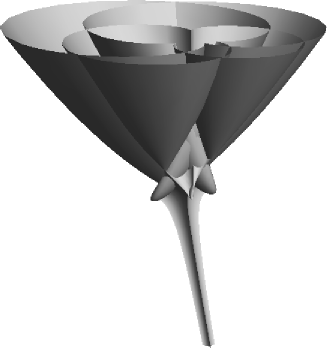}&
\includegraphics[width=55mm]{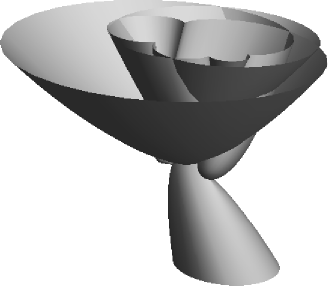}
\end{tabular}
\caption{Example \ref{ex:ms}  $\left(n = 3, m = 2, a=e^{\fr{7\pi i}{8}}\right)$}\label{Figure4}
\medskip
\shortstack{{\small
Left: 
$
c_1 = c_2 = c_3 = 0.
$
\quad
Right: 
$
c_1 = 1, c_2 = c_3 = 0.
$
}
\\
{\small
In both figures, the downward end corresponds to $z = \infty$.
}\\
{\small
In the right figure, the end $z = \infty$ is a complete end.
}
}

\end{center}
\end{figure}

\begin{remark}
We do not know whether or not there exists a complete, or weakly complete and of finite total curvature affine maxface with positive genus.
One can check that 
the Weierstrass data of the Chen--Gackstatter minimal surface (\cite{CG82}) does not give affine maxfaces since 
the period condition \eqref{maxperiod} is never satisfied.
Additionally, a partial construction of complete regular affine maximal maps of positive genus
without considering of the period condition was shown in \cite[Section 4]{AMM11}, 
but we also do not know explicit examples other than improper affine fronts with positive genus.
About complete improper affine fronts with genus one,
see \cite[Section 4]{Mart05_IAmap} and \cite[Section 4]{Matsu24}).
\end{remark}

\newcommand{\etalchar}[1]{$^{#1}$}

\end{document}